\pdfoutput=1
\documentclass[10pt,a4paper]{amsart}

\usepackage[utf8]{inputenc}

\usepackage{amssymb}
\usepackage{amsmath}
\usepackage{amsthm}

\usepackage{tikz}
\usepackage{pgfplots}
\pgfplotsset{compat=1.5}

\usepackage{enumitem}

\newlist{enumalpha}{enumerate}{1}
\setlist[enumalpha, 1]{
label=\alph*)
}

\usepackage[style=alphabetic,url=false]{biblatex}
\usepackage{hyperref}

\usepackage[english,noabbrev,sort]{cleveref}

\usepackage{algorithm}
\usepackage{algpseudocode}
\numberwithin{algorithm}{section}

\usepackage{my-math-definitions}

\newtheorem{theorem}{Theorem}
\newtheorem{lemma}[theorem]{Lemma}

\newtheorem{remark}[theorem]{Remark}
\newtheorem{corollary}[theorem]{Corollary}

\numberwithin{theorem}{section}
\numberwithin{equation}{section}

\author{Fabian Gundlach}
\title{Sampling cubic rings}

\newcommand{\V}{\mathcal V}
\newcommand{\W}{\mathcal W}
\DeclareMathOperator{\Reg}{Reg}

\subjclass{11Y40, 11R16}

\begin{document}

\begin{abstract}
We explain how to construct a uniformly random cubic integral domain $S$ of given signature with $|\disc(S)|\leq T$ in expected time $\widetilde\O(\log T)$.
\end{abstract}

\maketitle

\section{Introduction}

In the past decades, there has been an increasing amount of interest in the statistical properties of arithmetic objects such as number fields or orders in number fields.

In \cite{belabas-computing-cubic-fields}, Belabas gave an algorithm that computes a list of all cubic number fields $K$ with $|\disc(K)|\leq T$ in time $\widetilde\O(T)$. One can similarly enumerate all cubic integral domains $S$ (i.e., orders in cubic number fields) with $|\disc(S)|\leq T$ in time $\widetilde\O(T)$. The running time is essentially optimal since the number of such fields (or rings) is $\asymp T$ as was shown by Davenport and Heilbronn in \cite{davenport-heilbronn-cubic-fields}.

In this paper, we give an algorithm that constructs a uniformly random cubic integral domain $S$ of given signature with $|\disc(S)|\leq T$ in expected time $\widetilde\O(\log T)$. (See \Cref{cubicsummary} for the precise statement.) The running time is essentially optimal since the smallest possible total number of bits in the coefficients of a cubic form corresponding to such a random ring $S$ is on average $\asymp\log T$.

Like \cite{belabas-computing-cubic-fields} and \cite{davenport-heilbronn-cubic-fields}, we use Levi's parametrization of cubic rings by $\GL_2(\Z)$-orbits of integral binary cubic forms. (See \cite{levi-cubic}.)

The most straightforward idea (used in \cite{belabas-computing-cubic-fields} and \cite{davenport-heilbronn-cubic-fields}) for enumerating or counting the orbits with bounded discriminant is to construct a fundamental domain for the action of $\GL_2(\Z)$ on the space of binary cubic forms and to then enumerate or count the lattice points in this domain.

The unboundedness of the fundamental domain, i.e., the presence of (long and narrow) cusps presents an inconvenience when enumerating or counting the lattice points.
To deal with this issue, Bhargava introduced an elegant approach, which he calls ``averaging over fundamental domains'' or ``thickening and cutting off the cusps''. Very roughly speaking, he showed that it suffices to be able to estimate the number of points in linear transforms of a fixed region $U$ of our choice! But it is relatively easy to count lattice points in a linear transform of a fixed (large) ball, at least as long as the linear transformation does not deform the ball too much. His method was for example used to obtain more precise counting results in \cite{bhargava-shankar-tsimerman-cubic-counting}.

In this paper, we adapt his method to the problem of selecting orbits uniformly at random. Very roughly speaking, it suffices to be able to pick uniformly random lattice points from linear transforms of a fixed region $U$ of our choice and to have a good ``uniform'' upper bound on the number of lattice points in this transform. We use rejection sampling to ensure that all orbits occur with exactly the same probability. The quality of the uniform upper bound on the number of lattice points is important for estimating the running time of the algorithm.
To the author, it seems not quite straightforward to choose an element of the fundamental domain uniformly at random without the use of Bhargava's method.

An implementation of the algorithm as a standalone program is provided. (See \cref{implementation}.)

Our approach can be adapted to other arithmetic objects parameterized by prehomogeneous vector spaces such as the famous parameterization of ideal classes of quadratic rings by orbits of binary quadratic forms or the parameterizations given in \cite{bhargava-quartic-parametrization}, \cite{bhargava-quintic-parametrization}. (See \cref{outlook}.)

\paragraph{Acknowledgements}
This work was supported by the Deutsche Forschungsgemeinschaft (DFG, German Research Foundation) --- Project-ID 491392403 --- TRR 358, project A4.
The author is grateful for helpful conversations with Noam Elkies, Jürgen Klüners, and Anne-Edgar Wilke.

\section{Preparations}

It is easy to select an integer lattice point in an axis-parallel box with integer side lengths uniformly at random. This remains true if we apply a triangular linear transformation to the box:

\begin{lemma}\label{triangular_counting}
Consider an axis-parallel box $I=I_1\times\cdots\times I_n\subset\R^n$, where $I_1,\dots,I_n$ are half-open intervals, say $I_i=[a_i,a_i+l_i)$, of integer lengths $l_1,\dots,l_n\in\Z$. Let $M=(m_{ij})_{i,j}$ be a lower-triangular unipotent matrix with inverse $M^{-1}=(m_{ij}')_{i,j}$. Then, the set $\Z^n\cap MI$ has size $l_1\cdots l_n$ and the following algorithm selects an element $v=(v_1,\dots,v_n)$ of $\Z^n\cap MI$ uniformly at random.
\end{lemma}
\begin{algorithm}[H]
\caption{Finding a random lattice point in a transformed box}
\label{triangular_counting_alg}
\begin{algorithmic}[1]
\For{$i\gets1,\dots,n$}
\State Pick an element $\delta_i$ of $\{0,\dots,l_i-1\}$ uniformly at random.
\State $v_i \gets \left\lceil a_i - \sum_{j<i} m_{ij}'v_j\right\rceil + \delta_i$
\EndFor
\end{algorithmic}
\end{algorithm}
\begin{proof}
By definition, for any $v\in\R^n$, we have $v\in MI$ if and only if
\[
\sum_j m_{ij}'v_j \in [a_i,a_i+l_i)\qquad\textnormal{for all }i.
\]
The inverse matrix $M^{-1}$ is also lower-triangular and unipotent, so the sum on the left-hand side is $v_i + \sum_{j<i}m_{ij}'v_j$.

Hence, $v\in MI$ if and only if
\[
a_i - \sum_{j<i}m_{ij}'v_j \leq v_i < a_i - \sum_{j<i}m_{ij}'v_j + l_i\qquad\textnormal{for all }i.
\]
The integer solutions $v_i$ to these inequalities are
\[
v_i = \left\lceil a_i - \sum_{j<i} m_{ij}'v_j\right\rceil + \delta_i\qquad\textnormal{for }\delta_i\in\{0,\dots,l_i-1\}.
\qedhere
\]
\end{proof}

Our algorithm will work with approximations of random real numbers and we will show that a given precision suffices with large probability to decide whether a particular polynomial inequality $f(x)>0$ holds. The following lemma will help with this analysis. (Here, $\varepsilon$ will be the precision to which we have computed $f(x)$.)

\begin{lemma}\label{small_values}
Let $n\geq1$. For any monic polynomial $f\in\R[X]$ of degree $n$ and any $\varepsilon>0$, the set of $x\in\R$ with $|f(x)|\leq \varepsilon$ has measure at most $2n\varepsilon^{1/n}$.
\end{lemma}
\begin{proof}
Let $r_1,\dots,r_n$ be the complex roots of $f$, so that $f(X)=(X-r_1)\cdots(X-r_n)$. Any $x\in\R$ with $|f(x)|\leq\varepsilon$ must have distance at most $\varepsilon^{1/n}$ from one of the roots $r_i$. For any given root $r_i$, the set of $x\in\R$ with $|x-r_i|\leq\varepsilon^{1/n}$ forms an interval of length at most $2\varepsilon^{1/n}$. The claim follows by summing over all roots.
\end{proof}

\section{Cubic rings}

In this section, we describe an algorithm to construct random cubic rings.

We first remind the reader of some basic definitions and of Levi's parameterization of cubic rings by binary cubic forms. A \emph{cubic ring} is a (commutative unitary) ring which is isomorphic to $\Z^3$ as a $\Z$-module. A \emph{cubic integral domain} is a cubic ring which is an integral domain. They are exactly the orders in cubic number fields. The \emph{signature} of a cubic integral domain is the number of real embeddings of its field of fractions.

For any ring $R$, let $\V(R)$ denote the set of binary cubic forms $f=aX^3+bX^2Y+cXY^2+dY^3$ with $a,b,c,d\in R$. We can naturally identify $\V(R)$ with $R^4$ by identifying a cubic form $f$ with its coefficient vector $(a,b,c,d)$.

Define an action of $\GL_2(R)$ on $\V(R)$ by
\[
(Mf)(v) = \det(M)^{-1}\cdot f(M^Tv)\qquad\textnormal{for $M\in\GL_2(R)$ and $f\in\V(R)$ and $v\in R^2$}.
\]
The discriminant $\disc(f)$ is a homogeneous degree $4$ polynomial in the coefficients $a,b,c,d$. It is invariant under the action of the group
\[\SL_2^{\pm}(R) = \{g\in\GL_2(R):\det(g)=\pm1\}.
\]
In \cite{levi-cubic}, Levi described a bijection
\[
\GL_2(\Z)\backslash\V(\Z) \longleftrightarrow \{\textnormal{cubic rings}\}
\]
with the following properties:
\begin{enumalpha}
\item If an orbit $\GL_2(\Z)f$ corresponds to the cubic ring $S$, then $\disc(f) = \disc(S)$ and there is a group isomorphism $\Stab_{\GL_2(\Z)}(f)\cong\Aut(S)$.
\item The \emph{irreducible orbits} (orbits consisting of cubic forms $f$ that are irreducible over $\Q$) correspond to the cubic integral domains. If an irreducible orbit $\GL_2(\Z)f$ corresponds to the cubic integral domain $S$, then the signature of $S$ is the number of roots of $f$ in $\vP^1(\R)$.
\item Concretely, the cubic ring $S$ corresponding to an orbit $\GL_2(\Z)f$ with $f=aX^3+bX^2Y+cXY^2+dY^3$ has a $\Z$-basis of the form $(1,\omega_1,\omega_2)$ with
\begin{align*}
\omega_1\cdot\omega_2 &= -ad\cdot 1, \\
\omega_1\cdot\omega_1 &= -ac\cdot 1 - b\cdot\omega_1 + a\cdot\omega_2, \\
\omega_2\cdot\omega_2 &= -bd\cdot 1 - d\cdot\omega_1 + c\cdot\omega_2.
\end{align*}
\end{enumalpha}

Any element $g$ of $\SL_2^\pm(\R)$ can be written uniquely as a product $nak$ with
\[
n\in N(\R) = \left\{\mat{1\\t&1}:t\in\R\right\},
\quad
a\in A(\R) = \left\{\mat{s^{-1}\\&s}:s\in\R^\times\right\},
\quad
k\in O_2(\R).
\]
In fact, the resulting bijection $N(\R)\times A(\R)\times O_2(\R) \leftrightarrow \SL_2^\pm(\R)$ is a diffeomorphism. Let $\dd t,\dd^\times s,\dd^\times k$ be Haar measures on the groups $N(\R)\cong\R$, $A(\R)\cong\R^\times$, and $O_2(\R)$, respectively. Then, the pushforward of the measure $\dd t\cdot s^{-2}\dd^\times s\cdot\dd^\times k$ is a Haar measure $\dd^\times g$ on $\SL_2^\pm(\R)$.

Matrices of the form
\[
n(t):=\mat{1\\t&1}\qquad\textnormal{or}\qquad a(s):=\mat{s^{-1}\\&s}
\]
act on $\V(\R)\cong\R^4$ as
\[
\mat{1\\3t&1\\3t^2&2t&1\\t^3&t^2&t&1}\qquad\textnormal{or}\qquad\mat{s^{-3}\\&s^{-1}\\&&s\\&&&s^3}.
\]
Let $s_{\min}:=\sqrt{\sqrt{3}/2}$. The \emph{Siegel set}
\begin{equation}\label{defF}
\sF := \{n(t)a(s)k:|t|\leq\tfrac12,\ s\geq s_{\min},\ k\in O_2(\R)\} \subset \SL_2^\pm(\R)
\end{equation}
contains the famous \emph{Gauss set}
\begin{equation}\label{defFs}
\sF' := \{n(t)a(s)k: |t|\leq\tfrac12,\ s\geq \sqrt[4]{1-t^2},\ k\in O_2(\R)\} \subset \SL_2^\pm(\R),
\end{equation}
which is a fundamental domain for the left action of $\GL_2(\Z)$ on $\SL_2^\pm(\R)$. More precisely, we have
\[
\sum_{h\in\GL_2(\Z)}\alpha(hg) = 4 \textnormal{ for all }g\in\SL_2^\pm(\R),
\]
where $\alpha(g)=\alpha(t,s)$ for $g=n(t)a(s)k$ with
\begin{equation}\label{defalpha}
\begin{aligned}
\alpha(g) &= 0 &&\textnormal{if $g$ lies outside $\sF'$,}&&\textnormal{i.e., }|t|>\frac12\textnormal{ or }s<\sqrt[4]{1-t^2},\\
\alpha(g) &= 1 &&\textnormal{if $g$ lies in the interior of $\sF'$,}&&\textnormal{i.e., }|t|<\frac12\textnormal{ and }s>\sqrt[4]{1-t^2},\\
\alpha(g) &\in [0,1] &&\textnormal{always}.
\end{aligned}
\end{equation}
(For simplicity, we don't give the formulas for $\alpha(g)$ on the measure $0$ boundary of~$\sF'$. The sum is $4$ because for any matrix $g$ in $\sF'$, the matrices $\smat{\pm1\\&\pm1}g$ also lie in $\sF'$.)

We now fix a signature $r\in\{1,3\}$ and let $\V^r(\R)\subset\V(\R)$ be the set of squarefree cubic forms with exactly $r$ real roots. Let $\W^r(\R)\subset\V^r(\R)$ be the subset consisting of those cubic forms $f$ with $|\disc(f)|=1$. The group $\SL_2^\pm(\R)$ acts transitively on $\W^r(\R)$. The set $\W^r(\R)$ is a $3$-dimensional smooth submanifold of $\V(\R)$. It follows that the pushforward of $\dd^\times g$ along the map $\SL_2^\pm(\R)\ra\W^r(\R)$, $g\mapsto g^{-1}f_0$, is independent of the choice of $f_0\in\W^r(\R)$. Denote this pushforward by $\dd f$. Bhargava's averaging over fundamental domains / thickening of cusps trick (see for example \cite[section 2.2]{bhargava-quartic-counting} or \cite[section 5.3]{bhargava-shankar-tsimerman-cubic-counting}) relies on the following lemma:
\begin{lemma}\label{thickening_lemma}
Fix an integrable subset $U\subseteq\W^r(\R)$, we define the function $\eta$ on $\W^r(\R)$ to be the following average of the indicator functions of linear transforms of $U$ by elements $g$ of the fundamental domain $\sF'$:
\[
\eta(f) = \int_{\sF'}\chi_{gU}(f)\dd^\times g
\]
and let
\[
C = \int_{\W^r(\R)}\chi_U(f)\dd f.
\]
We then have for all $f\in\W^r(\R)$:
\[
\sum_{h\in\GL_2(\Z)}
\eta(h f)
= 4C.
\]
\end{lemma}
\begin{proof}
We have
\begin{align*}
\sum_{h\in\GL_2(\Z)}\eta(hf)
&= \sum_{h\in\GL_2(\Z)}\int_{\sF'}\chi_{gU}(hf)\dd^\times g \\
&= \sum_{h\in\GL_2(\Z)}\int_{\SL_2^\pm(\R)}\alpha(g)\chi_{gU}(hf)\dd^\times g \\
&= \sum_{h\in\GL_2(\Z)}\int_{\SL_2^\pm(\R)}\alpha(g)\chi_{h^{-1}gU}(f)\dd^\times g\\
&= \sum_{h\in\GL_2(\Z)}\int_{\SL_2^\pm(\R)}\alpha(hg)\chi_{gU}(f)\dd^\times g \\
&= 4\int_{\SL_2^\pm(\R)}\chi_{gU}(f)\dd^\times g
= 4\int_{\SL_2^\pm(\R)}\chi_U(g^{-1}f)\dd^\times g\\
&= 4\int_{\W^r(\R)}\chi_U(f)\dd f
= 4C.
\qedhere
\end{align*}
\end{proof}

Let $\|\cdot\|$ be the norm on $\V(\R)$ associated to the following positive definite quadratic form $q$ on $\V(\R)$:
\[
q(aX^3+bX^2Y+cXY^2+dY^3) = 5a^2+b^2+c^2+5d^2+2ac+2bd.
\]

\begin{lemma}
The quadratic form $q$ is positive definite and invariant under the action of the orthogonal group $O_2(\R)\subset\GL_2(\R)$.
\end{lemma}
\begin{proof}
One can check that
\[
\int_{\{(x,y)\in\R^2:x^2+y^2\leq1\}} f(x,y)^2\dd(x,y)
= \frac{\pi}{64}\cdot q(f).
\]
The left-hand side is clearly positive definite and $O_2(\R)$-invariant.
\end{proof}

\begin{lemma}\label{in_box}
Consider the closed unit ball of radius $1$ in $\V(\R)$. The maximum values of $|a|,|b|,|c|,|d|$ on this set are $\frac12,\frac{\sqrt5}{2},\frac{\sqrt5}{2},\frac12$, respectively.
\end{lemma}
\begin{proof}
To achieve the maximum value of $|a|$, for instance, the derivatives of $q(f)$ with respect to $b,c,d$ need to vanish. One easily checks that this occurs exactly at the two points $\pm\frac12(X^3-XY^2)$.

The maximum for $|b|$ occurs at $\pm\frac{\sqrt5}{2}(X^2Y-\frac15Y^3)$. The maxima for $|c|$ and $|d|$ follow by symmetry.
\end{proof}

Consider the open unit ball in $\V(\R)$:
\[
B = \{f\in\V(\R):\|f\|<1\}
\]
By \Cref{in_box}, $B$ is contained in the box $I:=I_1\times\cdots\times I_4$ given by
\begin{align*}
I_1 &:= \textstyle\{a\in\R:|a|\leq\frac12\}, & I_2 &:= \textstyle\{b\in\R:|b|\leq\frac{\sqrt5}2\},\\
I_3 &:= \textstyle\{c\in\R:|c|\leq\frac{\sqrt5}2\}, & I_4 &:= \textstyle\{d\in\R:|d|\leq\frac12\}.
\end{align*}
Hence, for any $\lambda,s>0$, the set $\lambda\cdot a(s)B$ is contained in the box $\lambda\cdot a(s)I$ with side lengths
\begin{align*}
l_1(\lambda,s) &:= \lambda s^{-3}, &
l_2(\lambda,s) &:= \sqrt5 \lambda s^{-1}, \\
l_3(\lambda,s) &:= \sqrt5 \lambda s, &
l_4(\lambda,s) &:= \lambda s^{3}.
\end{align*}
The box $\lambda\cdot a(s)I$ is in turn contained in the box $I'(\lambda,s)=I_1'(\lambda,s)\times\cdots\times I_4'(\lambda,s)$ with integer side lengths $l_i'(\lambda,s)$, where
\[
I_i'(\lambda,s):=\left[-\tfrac12l_i'(\lambda,s),\tfrac12l_i'(\lambda,s)\right)\textnormal{ with }l_i'(\lambda,s) := \lfloor1+l_i(\lambda,s)\rfloor\in\Z\textnormal{ for }i=1,\dots,4.
\]
Let $s_{\max}=\sqrt[3]{\lambda/2}$. For $s_{\min}\leq s\leq s_{\max}$, we can bound the side lengths $l_i'(\lambda,s)$ from above as follows:
\begin{equation}\label{floorbound}
\begin{aligned}
l_1'(\lambda,s)&\leq L_1'(\lambda)\cdot s^{-3}, &
l_2'(\lambda,s)&\leq L_2'(\lambda)\cdot s^{-1}, \\
l_3'(\lambda,s)&\leq L_3'(\lambda)\cdot s, &
l_4'(\lambda,s)&\leq L_4'(\lambda)\cdot s^{3},
\end{aligned}
\end{equation}
where
\begin{align*}
L_1'(\lambda) &:= s_{\max}^3+\lambda, &
L_2'(\lambda) &:= s_{\max}+\sqrt5\lambda, \\
L_3'(\lambda) &:= s_{\min}^{-1}+\sqrt5\lambda, &
L_4'(\lambda) &:= s_{\min}^{-3}+\lambda.
\end{align*}

Fix a radius $R>0$ large enough so that $R\cdot B\cap \W^r(\R)\neq\emptyset$. (For example, you can take $R=7/4$ for $r=1$ and $R=5/4$ for $r=3$.)

We now give an algorithm for computing random irreducible $\GL_2(\Z)$-orbits. This first algorithm uses real arithmetic and we later explain how to deal with precision issues.

\begin{theorem}\label{cubic-thm}
Let $T\geq1$ and $\lambda=RT^{1/4}$. Let $s_{\min} = \sqrt{\sqrt{3}/2}$ and $s_{\max} = \sqrt[3]{\lambda/2}$.
\begin{enumalpha}
\item The following algorithm either fails or produces an element $f$ of an irreducible $\GL_2(\Z)$-orbit in $\V^r(\Z)$ with $0<|\disc(f)|\leq T$.
\item The probability of returning an element of any given such orbit $\GL_2(\Z)f$ is (for fixed $r$ and $T$) proportional to $1/\#\Stab_{\GL_2(\Z)}(f)$.
\item\label{cubic-prob} The probability of success is $>p_0$ for a constant $p_0$ (depending only on $r$ and $R$, but not on $T$) if there is at least one such orbit $\GL_2(\Z)f$.
\end{enumalpha}
\end{theorem}
\begin{algorithm}[H]
\caption{Finding a random orbit of cubic forms}\label{random_cubic_alg}
\begin{algorithmic}[1]
\State Pick an element $t$ of $(-\frac12,\frac12)$ uniformly at random.
\State Pick an element $s$ of $(s_{\min},\infty)$ at random with probability measure proportional to $s^{-2}\dd^\times s$.
\State \Return FAIL \textbf{unless} $\sqrt[4]{1-t^2}<s<s_{\max}$.\label{checks1}
\State \Return FAIL \textbf{with probability} $1-\frac{l_1'(\lambda,s)\cdots l_4'(\lambda,s)}{L_1'(\lambda)\cdots L_4'(\lambda)}$.\label{probline}
\State Pick a point $f\in\V(\Z)\cap n(t)I'(\lambda,s)$ uniformly at random using \Cref{triangular_counting_alg}.\label{pickline}
\State \Return FAIL \textbf{unless} $f\in\V^r(\R)$ and $0<|\disc(f)|\leq T$ and
and $f$ is irreducible over $\Q$ and
$f \in R|\disc(f)|^{1/4}\cdot n(t)a(s)B$.
\State \Return $f$.
\end{algorithmic}
\end{algorithm}
\begin{remark}\label{repeat}
Repeat the algorithm until it succeeds. According to \ref{cubic-prob}, if there is such an orbit, it will succeed after at most $1/p_0$ attempts on average.
\end{remark}
\begin{remark}\label{smallest}
The smallest absolute discriminant of an irreducible orbit with signature $r=3$ is $49$. The smallest absolute discriminant of an irreducible orbit with signature $r=1$ is $23$. (See for example tables of number fields of degree $3$ and small discriminant, such as the one available at \cite{lmfdb}.)
\end{remark}
\begin{remark}\label{uniform}
The following algorithm satisfies a) and c), but every orbit has the same probability of being returned: Compute a random orbit $\GL_2(\Z)f$ using \Cref{random_cubic_alg}. Return this orbit $\GL_2(\Z)f$ with probability $\#\Stab_{\GL_2(\Z)}(f)/3$ and return FAIL otherwise. (If $f$ corresponds to the cubic integral domain $S$ with field of fractions $K$, the group $\Stab_{\GL_2(\Z)}(f)\cong\Aut(S)\subseteq\Aut(K)$ is either trivial or cyclic of order $3$.) 
\end{remark}
\begin{proof}[Proof of \Cref{cubic-thm}]
For $s_{\min}\leq s\leq s_{\max}$, we have
\[
0 \leq \frac{l_1'(\lambda,s)\cdots l_4'(\lambda,s)}{L_1'(\lambda)\cdots L_4'(\lambda)} \leq1,
\]
according to (\ref{floorbound}), so line \ref{probline} of the algorithm makes sense. Claim a) is clear.

Now, consider any irreducible cubic form $f\in\V^r(\Z)$ with $0<|\disc(f)|\leq T$.

Since $|t|\leq\frac12$, we have $s_{\min}=\sqrt{\sqrt3/2}\leq\sqrt[4]{1-t^2}$.
The algorithm returns the cubic form $f$ with probability
\[
P(f) :=
\frac{
\int_{-1/2}^{1/2} \int_{\sqrt[4]{1-t^2}}^{s_{\max}}
p(f,t,s)\cdot
s^{-2}\dd^\times s\,\dd t
}
{
\int_{-1/2}^{1/2}\int_{s_{\min}}^\infty s^{-2}\dd^\times s\,\dd t
},
\]
where
\[
p(f,t,s) := 
\frac{l_1'(\lambda,s)\cdots l_4'(\lambda,s)}{L_1'(\lambda)\cdots L_4'(\lambda)}\cdot
\frac{\chi_{n(t)I'(\lambda,s)}(f)}{\#(\V(\Z)\cap n(t)I'(\lambda,s))}\cdot
\chi_{R|\disc(f)|^{1/4}\cdot n(t)a(s)B}(f).
\]
We have $\#(\V(\Z)\cap n(t)I'(\lambda,s)) = l_1'(\lambda,s)\cdots l_4'(\lambda,s)$ by \Cref{triangular_counting}.

Since $0<|\disc(f)|\leq T$ and $\lambda=RT^{1/4}$, we moreover have
\begin{equation}\label{inclusions}
R|\disc(f)|^{1/4}\cdot n(t)a(s)B \subseteq\lambda\cdot n(t)a(s)B\subseteq\lambda\cdot n(t)a(s)I\subseteq n(t)I'(\lambda,s).
\end{equation}
Hence,
\[
p(f,t,s) = \frac{\chi_{R|\disc(f)|^{1/4}\cdot n(t)a(s)B}(f)}{L_1'(\lambda)\cdots L_4'(\lambda)}.
\]
Moreover, since $f=aX^3+\cdots+dY^3$ is irreducible, we have $a\neq0$, so $|a|\geq1$. On the other hand, if $p(f,t,s)\neq0$, then by (\ref{inclusions}), we have $f\in\lambda\cdot n(t)a(s)I$, which implies $|a|\leq\frac12 \lambda s^{-3}$. Therefore, if $p(f,t,s)\neq0$, then $1\leq\frac12 \lambda s^{-3}$, so $s\leq\sqrt[3]{\lambda/2}=s_{\max}$. (Bhargava calls this inequality ``cutting off the cusp''.) Hence, the upper bound on $s$ in the integral in the numerator of $P(f)$ can be omitted without changing the value of the integral:
\begin{align*}
P(f) &=
\frac{
\int_{-1/2}^{1/2} \int_{\sqrt[4]{1-t^2}}^{\infty}
p(f,t,s)\cdot
s^{-2}\dd^\times s\,\dd t
}
{
\int_{-1/2}^{1/2}\int_{s_{\min}}^\infty s^{-2}\dd^\times s\,\dd t
}. 
\end{align*}
We now multiply both the numerator and the denominator by the finite number $\int_{O_2(\R)}\dd^\times k$. Since $B$ is $O_2(\R)$-invariant and since the measure $\dd t\cdot s^{-2}\dd^\times s\cdot \dd^\times k$ on $N(\R)\times A(\R)\times O_2(\R)$ corresponds to the Haar measure $\dd^\times g$ on $\SL_2^\pm(\R)$, it follows together with the definitions of $\sF$ and $\sF'$ in (\ref{defF}) and (\ref{defFs}) that
\[
P(f) = \frac{\int_{\sF'}p(f,g)\dd^\times g}{\int_{\sF}\dd^\times g}
\]
with
\[
p(f,g) := \frac{
\chi_{R|\disc(f)|^{1/4}\cdot gB}(f)}{L_1'(\lambda)\cdots L_4'(\lambda)}.
\]
Note that $|\disc(|\disc(f)|^{-1/4}\cdot f)|=1$ and therefore $|\disc(f)|^{-1/4}\cdot f\in\W^r(\R)$ because $\disc(f)$ is a homogeneous polynomial of degree $4$ in the coefficients of $f$. Setting $U:=R\cdot B\cap\W^r(\R)$, we can therefore rewrite $p(f,g)$ as
\[
p(f,g) := \frac{
\chi_{gU}(|\disc(f)|^{-1/4}\cdot f)}{L_1'(\lambda)\cdots L_4'(\lambda)},
\]
so
\[
P(f) =
\frac{
\int_{\mathcal F'}
\chi_{gU}(|\disc(f)|^{-1/4}\cdot f)\cdot
\dd^\times g
}
{
L_1'(\lambda)\cdots L_4'(\lambda)\cdot \int_{\mathcal F} \dd^\times g
}
\]
By \Cref{thickening_lemma} applied to the element $|\disc(f)|^{-1/4}\cdot f$ of $\W^r(\R)$, we obtain
\[
\sum_{h\in\GL_2(\Z)}P(hf) = D(\lambda) :=
\frac{C}{L_1'(\lambda)\cdots L_4'(\lambda)},
\]
with
\[
C = \frac{4\cdot\int_{\W^r(\R)}\chi_{U}(f')\dd f'}{\int_\sF\dd^\times g}.
\]
The constant $C$ is positive because $U$ is a nonempty open subset of $\W^r(\R)$ and $\int_{\sF}\dd^\times g<\infty$.

On the other hand,
\[
\sum_{h\in\GL_2(\Z)}P(hf) = \#\Stab_{\GL_2(\Z)}(f)\cdot\sum_{f'\in\GL_2(\Z)f}P(f').
\]
We conclude that for any irreducible cubic form $f\in\V^r(\Z)$ with $0<|\disc(f)|\leq T$, we have the following probability of returning an element of its $\GL_2(\Z)$-orbit:
\[
\sum_{f'\in\GL_2(\Z)f} P(f') = \frac{D(\lambda)}{\#\Stab_{\GL_2(\Z)}(f)}.
\]
Since $D(\lambda)$ is independent of $f$, this proves b).

For c), we sum over all $\GL_2(\Z)$-orbits of irreducible cubic forms $f\in\V^r(\Z)$ with $0<|\disc(f)|\leq T$. The number of such orbits is $\asymp T$ for $T\ra\infty$. (See \cite{davenport-heilbronn-cubic-fields}.) In particular, it is $\gg T$ as long as there is at least one such orbit. Since $\Stab_{\GL_2(\Z)}(f)$ has size at most $3$ by \Cref{uniform}, the sum of $1/\#\Stab_{\GL_2(\Z)}(f)$ over all orbits is also $\gg T$.

By definition, $L_i'(\lambda)\ll\lambda$ for all $i$, so $D(\lambda)\gg\lambda^{-4}\gg T^{-1}$.

It follows that the probability of success is
\begin{align*}
\sum_{\substack{f\in\V^r(\Z)\\\textnormal{irreducible}\\0<|\disc(f)|\leq T}}P(f)
&= \sum_{\substack{[f]\in\GL_2(\Z)\backslash\V^r(\Z)\\\textnormal{irreducible}\\0<|\disc(f)|\leq T}}\sum_{f'\in\GL_2(\Z)f}P(f') \\
&= \sum_{\substack{[f]\in\GL_2(\Z)\backslash\V^r(\Z)\\\textnormal{irreducible}\\0<|\disc(f)|\leq T}} \frac{D(\lambda)}{\#\Stab_{\GL_2(\Z)}(f)} \gg T\cdot T^{-1} = 1.
\end{align*}
This proves c).
\end{proof}

The above algorithm uses real arithmetic. In practice, we need to work with approximations of those real numbers to finite precision. Note that we need to be able to decide the inequalities appearing in the algorithm and to compute $\lfloor x\rfloor$ or $\lceil x\rceil$ for various real numbers $x$ appearing in the algorithm. By analyzing the probability with which a given precision suffices to decide the inequalities and to compute the floor and ceiling values, we show:

\begin{theorem}\label{cubicimpl}
\Cref{random_cubic_alg} can be implemented on a random access machine (with a random bit generator) with expected running time $\widetilde\O(\log T)$.
\end{theorem}
\begin{proof}
We use the well-known method of doubling the precision until it suffices. See \Cref{random_cubic_alg_bit} below for the concrete implementation.

We leave it to the reader to verify that \Cref{random_cubic_alg_bit} is functionally equivalent to \Cref{random_cubic_alg}. (Line \ref{checks} of \Cref{random_cubic_alg_bit} corresponds to line \ref{checks1} of \Cref{random_cubic_alg}. Lines \ref{compa} and \ref{compb}--\ref{compd} correspond to line \ref{pickline}.)

We show that the probability that the algorithm doesn't finish within the first $i$-th iterations (due to insufficient precision) is $\O(T^v\varepsilon^u)=\O(T^v2^{-u2^i})$ for constants $u,v>0$ and that the running time of the $i$-th iteration is $\widetilde\O(\log T+p)=\widetilde\O(\log T+2^i)$. The expected total running time is then
\begin{align*}
&{}\ll \sum_{i\geq1}\min(1,T^v2^{-u2^i})\cdot\widetilde\O(\log T+2^i) \\
&{}\ll \sum_{1\leq i\leq \lfloor\log_2(\frac vu\log_2T)\rfloor}\widetilde\O(\log T+2^i)
+ \sum_{i>\lfloor\log_2(\frac vu\log_2T)\rfloor} T^v2^{-u2^i}\cdot\widetilde\O(\log T+2^i) \\
&{}\ll \widetilde\O(\log T) + \sum_{j\geq\lfloor\frac vu\log_2T\rfloor}T^v2^{-uj}\cdot\widetilde\O(\log T+j) \\
&{}\ll \widetilde\O(\log T) + \sum_{j\geq0}2^{-uj}\cdot\widetilde\O(\log T+j) \\
&{}\ll \widetilde\O(\log T)
\end{align*}
as claimed.

All real numbers computed in the algorithm are $\ll T^{\O(1)}$. It follows that the real numbers in the $i$-th iteration can be computed to absolute precision $\O(T^{\O(1)}\varepsilon)$ in time $\widetilde\O(\log T+p)$.

Now, we explain the failure probabilities and running times of the individual steps of the algorithm:

\begin{description}
\item[Line \ref{checks}]
The probability that the inequalities in line \ref{checks} cannot be decided from the given approximation is $\O(T^{\O(1)}\varepsilon)$ since both sides of the inequalities are known to an absolute precision of $\O(T^{\O(1)}\varepsilon)$ and the probability that the values are within this distance is $\O(T^{\O(1)}\varepsilon)$ as $\sigma$ is uniformly distributed on $(0,1)$.
\item[Line \ref{checkpi}]
This works similarly.
\item[Line \ref{compl}]
Being able to compute $l_i'(\lambda,s)=\lfloor1+l_i(\lambda,s)\rfloor$ from the approximation of $1+l_i(\lambda,s)$ is equivalent to $l_i(\lambda,s)$ not lying within a distance of $\O(T^{\O(1)}\varepsilon)$ from any integer $k>0$. The closest integer to $l_i(\lambda,s)$ is $\ll T$, so we only need to consider $\ll T$ such integers $k$. We first bound the probability that $|l_4(\lambda,s)-k|=|\lambda s^3-k|\ll T^{\O(1)}\varepsilon$. As $\lambda\asymp T^{\O(1)}$, this is equivalent to $|s^3-\lambda^{-1}k|\ll T^{\O(1)}\varepsilon$. By \Cref{small_values}, the set of such values $s\in\R$ has Lebesgue measure $\O(T^{\O(1)}\varepsilon^{1/3})$. As $s$ is bounded from below (by $\sqrt{\sqrt3/2}$), the probability measure (proportional to $s^{-2}\dd^\times s$) of this set of values $s$ is also $\O(T^{\O(1)}\varepsilon^{1/3})$. Summing over the $\ll T$ values $k$, we see that the probability that the approximation is insufficient for computing $l_4(\lambda,s)$ is $\O(T^{\O(1)}\varepsilon^{1/3})$. A similar argument works for $l_3(\lambda,s)$. For $l_1(\lambda,s)$, we instead use that the inequality $|l_1(\lambda,s)-k|=|\lambda s^{-3}-k|\ll T^{\O(1)}\varepsilon$ is equivalent to $|s^3-\lambda k^{-1}|\ll T^{\O(1)}\varepsilon$ since $1\leq k\ll T$ and $1\ll s\ll T^{\O(1)}$. We then proceed as before. The same argument works for $l_2(\lambda,s)$.
\item[Line \ref{compdelta}]
Since $l_i'(\lambda,s)\ll T^{\O(1)}$ and $\Delta_i$ is chosen uniformly at random from $(0,1)$, the probability that $\Delta_i\cdot l_i'(\lambda,s)$ lies within a distance $\O(T^{\O(1)}\varepsilon)$ from an integer is at most $\O(T^{\O(1)}\varepsilon)$.
\item[Line \ref{compa}]
The number $\lceil-\frac12l_1'(\lambda,s)\rceil$ can be computed with integer arithmetic since $l_1'(\lambda,s)\in\Z$.
\item[Lines \ref{compb}--\ref{compd}]
We compute $\lceil x\rceil$ for various numbers $x$. Note that in each line, $x$ is a polynomial in $t$ of degree at most $3$ whose leading coefficient has absolute value $\geq1$. By the same argument as for line \ref{compl}, the probability that any of these values cannot be computed is $\O(T^{\O(1)}\varepsilon^{1/3})$.
\item[Line \ref{checkdisc}]
This involves only integer arithmetic.
\item[Line \ref{checksig}]
It suffices to check that $\disc(f)>0$ if $r=3$ and $\disc(f)<0$ if $r=1$.
\item[Line \ref{compq}]
If
\[q(a(s)^{-1}n(-t)f) - R^2|\disc(f)|^{1/2} \ll T^{\O(1)}\varepsilon,
\]
then
\[s^6\cdot (q(a(s)^{-1}n(-t)f) - R^2|\disc(f)|^{1/2})\ll T^{\O(1)}\varepsilon.
\]
The left-hand side is a polynomial in $s$ of degree $12$ with leading coefficient $5a^2\geq5$. By \Cref{small_values}, the Lebesgue measure of the set of values $s$ satisfying the inequality is therefore $\O(T^{\O(1)}\varepsilon^{1/12})$. As $1\ll s\ll T^{\O(1)}$, the probability measure of the corresponding set of values $s$ is also $\O(T^{\O(1)}\varepsilon^{1/12})$.
\item[Line \ref{checkirred}]
Note that a cubic form $f\in\V(\Z)$ is reducible over $\Q$ if and only if it has a rational root. For $a\neq0$, this is equivalent to the polynomial $f(X,1)\in\Z[X]$ of degree $3$ having a rational root. The denominator of such a root must divide $a\ll T$. Hence, it suffices to check whether $f(X,1)$ has a root in $\frac1a\Z$. The real roots of $f(X,1)$ can be approximated with accuracy $\frac1a$ for example using a binary search and Sturm sequences in time $\widetilde\O(\log T)$.
\qedhere
\end{description}
\end{proof}
\begin{algorithm}[H]
\caption{Finding a random orbit of cubic forms (using bit operations)}
\label{random_cubic_alg_bit}
\begin{algorithmic}[1]
\For{$i\gets1,2,\dots$}
\State Let $p=2^i$ and $\varepsilon=2^{-p}$. In this iteration, we will compute all occurring real numbers to absolute precision $\O(T^{\O(1)}\varepsilon)$. If the computed precision is insufficient to decide an occurring inequality or to compute an occurring floor or ceiling value, we immediately go to the next iteration, starting over with the next value of $i$.
\State Pick uniformly random elements $\tau,\sigma,\pi,\Delta_1,\dots,\Delta_4$ of $(0,1)$ to absolute precision $\varepsilon$ by picking the first $p$ binary digits of each of the numbers, keeping any digits that were already picked in the previous iteration.
\State Compute $t=\tau-\frac12$.
\State \Return FAIL \textbf{unless} $\frac{s_{\min}^2}{s_{\max}^2} < \sigma < \frac{s_{\min}^2}{\sqrt{1-t^2}}$.\label{checks}
\State Compute $s=\frac{s_{\min}}{\sqrt{\sigma}}$.
\State Compute $l_i'(\lambda,s)=\lfloor1+l_i(\lambda,s)\rfloor$ for $i=1,\dots,4$.\label{compl}
\State \Return FAIL \textbf{if} $\pi>\frac{l_1'(\lambda,s)\cdots l_4'(\lambda,s)}{L_1'(\lambda)\cdots L_4'(\lambda)}$.\label{checkpi}
\State Compute $\delta_i=\lfloor\Delta_i\cdot l_i'(\lambda,s)\rfloor$ for $i=1,\dots,4$.\label{compdelta}
\State Compute $a=\lceil-\frac12l_1'(\lambda,s)\rceil+\delta_1$.\label{compa}
\State \Return FAIL \textbf{unless} $a\neq0$.
\State Compute $b=\lceil-\frac12l_2'(\lambda,s)+3ta\rceil+\delta_2$.\label{compb}
\State Compute $c=\lceil-\frac12l_3'(\lambda,s)-3t^2a+2tb\rceil+\delta_3$.
\State Compute $d=\lceil-\frac12l_4'(\lambda,s)+t^3a-t^2b+tc\rceil+\delta_4$.\label{compd}
\State Let $f=aX^3+bX^2Y+cXY^2+dY^3$.
\State \Return FAIL \textbf{unless} $0<|\disc(f)|\leq T$.\label{checkdisc}
\State \Return FAIL \textbf{unless} $f\in\V^r(\R)$.\label{checksig}
\State \Return FAIL \textbf{unless} $q(a(s)^{-1}n(-t)f) < R^2|\disc(f)|^{1/2}$.\label{compq}
\State \Return FAIL \textbf{unless} $f$ is irreducible over $\Q$.\label{checkirred}
\State \Return $f$.
\EndFor
\end{algorithmic}
\end{algorithm}

\begin{corollary}\label{cubicsummary}
There are algorithms which for a given number $r\in\{1,3\}$ and $T$ with
\[
T \geq
\begin{cases}
49,& r=3,\\
23,& r=1
\end{cases}
\]
compute in expected time $\widetilde\O(\log T)$ a cubic integral domain $S$ of signature $r$ with $|\disc(S)|\leq T$ and such that
\begin{enumalpha}
\item the probability of returning any given such ring $S$ is for fixed $T$ proportional to $1/\#\Aut(S)$ or
\item all such rings occur with the same probability.
\end{enumalpha}
(Here, cubic rings are represented by a corresponding binary cubic form.)
\end{corollary}
\begin{proof}
Claim a) follows immediately from \Cref{cubic-thm,cubicimpl} and \Cref{repeat,smallest}. For b), we also use \Cref{uniform}.
\end{proof}

\begin{remark}
To compute a random cubic number field $K$ with signature $r$ and $|\disc(K)|\leq T$ (either uniformly or with probability proportional to $1/\#\Aut(K)$), one can compute cubic integral domains $S$ until finding one which is the ring of integers of its field of fractions $K$. There are well-known algorithms for testing whether $S$ is the ring of integers of its field of fractions, whose running time is dominated by the factorization of the integer $\disc(S)$. (See for example \cite[section 6.1]{cohen-computational-number-theory-1}.)

In \cite{bach-random-factored-numbers}, Bach gave an algorithm that generates a uniformly random integer $1\leq x\leq N$ together with its factorization. The expected running time of his algorithm is that required for $\O(\log N)$ primality tests of numbers $1\leq p\leq N$. It could be an interesting problem to efficiently find a uniform random cubic integral domain together with the factorization of its discriminant.
\end{remark}

\section{Implementation}\label{implementation}

An implementation of the algorithm described above is available at \url{https://github.com/fagu/random-orbits}. It makes use of the FLINT library \cite{flint} for large integer and polynomial arithmetic, and in particular for arbitrary-precision interval arithmetic \cite{Johansson2017arb}.

\Cref{running_time_plot} gives the approximate average time it takes to generate one random cubic integral domain $S$ with signature $r$ and $|\disc(S)|\leq 2^t$, where the probability of obtaining a given ring $S$ is proportional to $1/\#\Aut(S)$. For example, it takes roughly $10^{-3}$~seconds to generate a random ring with $r=3$ and $|\disc(S)|\leq2^{200}$ and roughly one second to generate a random ring with $r=3$ and $|\disc(S)|\leq 2^{200\,000}$.

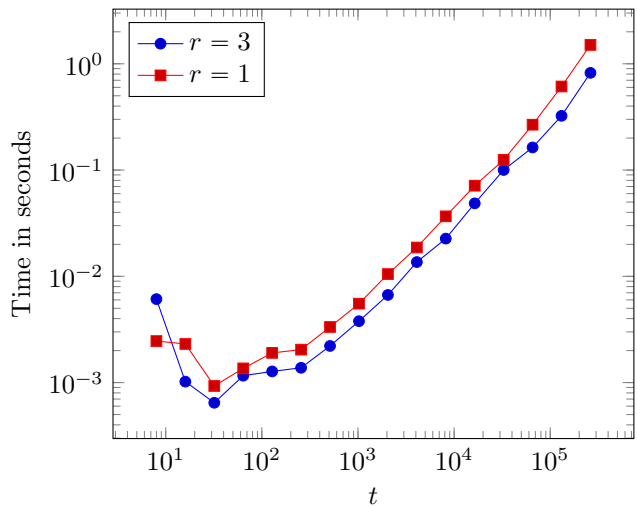
\begin{figure}[h]
\begin{tikzpicture}
\begin{loglogaxis}[xlabel={$t$}, ylabel={Time in seconds}, legend entries={$r=3$, $r=1$}, legend pos=north west]
\addplot table[x index=0, y index=1] {timing.dat};
\addplot table[x index=0, y index=2] {timing.dat};
\end{loglogaxis}
\end{tikzpicture}
\caption{Average time to generate one random ring}\label{running_time_plot}
\end{figure}

\section{Other parameterizations}\label{outlook}

The method described above can be adapted to other parameterizations by prehomogeneous vector spaces.

For example, it is well-known that there is a bijection between the set of $\GL_2(\Z)$-orbits of primitive binary quadratic forms and the set of ideal classes of orders in quadratic number fields. One obtains an algorithm which constructs a random pair $(S,I)$, where $S$ is a quadratic integral domain with given signature and $|\disc(S)|\leq T$ and $I$ is an invertible ideal class of $S$ with expected running time $\widetilde\O(\log T)$. The probability the algorithm returns a given such pair $(S,I)$ is proportional to $\frac{\Reg(S)}{\omega_S}$, where $\Reg(S)$ is the regulator of $S$ and $\omega_S$ is the number of roots of unity in $S$.

The method can also be used to construct random quartic or quintic rings as in \cite{bhargava-quartic-parametrization}, \cite{bhargava-quintic-parametrization} (with bounds derived as in \cite{bhargava-quartic-counting} and \cite{bhargava-quintic-counting}). Unfortunately, the constant factor in the running time crucially depends on the dimension of the prehomogeneous vector space. Quintic rings are parameterized by orbits in a $40$-dimensional prehomogeneous vector space and the constant factor becomes impractically large.

\printbibliography

\end{document}